\documentclass[12pt,twoside]{amsart}\usepackage{amssymb}
\usepackage{enumerate}

\newtheorem{thmspec}{\relax}

\newtheorem{theorem}{Theorem}[section]
\newtheorem{thm}[theorem]{Theorem}
\newtheorem{lem}[theorem]{Lemma}
\newtheorem{cor}[theorem]{Corollary}
\newtheorem{prop}[theorem]{Proposition}
\newtheorem{defi}[theorem]{Definition}

\theoremstyle{definition}

\theoremstyle{remark}

\numberwithin{equation}{section}
\def \Bbb{\mathbb}

\def\onto{{\kern3pt\to\kern-8pt\to\kern3pt}}

\def\<{\langle}
\def\>{\rangle}
\def\|{{\ |\ }}

\def\onto{\twoheadrightarrow}

\def\-{\underline}

\def\N{\Bbb N}

\def\R{\Bbb R}
\def\C{\Bbb C}
\def\P{\Bbb P}
\def\B{\Bbb B}

\def\codim{\operatorname{codim}}
\def\Id{\operatorname{Id}}
\def\d{\operatorname{d}}
\def\ddc{\operatorname{dd^c}}

\def\I{\mathcal{I}}

\def\<{\langle}
\def\>{\rangle}

\catcode`\@=11
\def\serieslogo@{\relax}
\def\@setcopyright{\relax}
\catcode`\@=12

\title[Green currents]
{Green currents  for quasi-algebraically stable   meromorphic self-maps of $\P^k$}

\begin{document}

\author{Vi{\^e}t-Anh  Nguy\^en}
 
 \address{
Math{\'e}matique-B{\^a}timent 425\\
UMR 8628, 15 rue Georges Cl\'emenceau\\
Universit{\'e} Paris-Sud\\
F-91405 Orsay, France}
\email{VietAnh.Nguyen@math.u-psud.fr}
\urladdr{http://www.math.u-psud.fr/$\sim$vietanh/}

\date{}


\begin{abstract}
We construct a canonical  Green current $T_f$ for every  quasi-algebraically stable  meromorphic self-map $f$ of $\P^k$
such that its   first dynamical degree  $\lambda_1(f)$  is a  simple root of its characteristic   polynomial     and
that  $\lambda_1(f)>1.$
We establish  a functional  equation for $T_f$ and  show that the support of  $T_f$ is  contained in the Julia set, which is  thus  non empty.
 \end{abstract}
\maketitle

\section{ Introduction}

    Let $f: \P^k\longrightarrow \P^k$ be a meromorphic self-map.   
    Then  there are    homogeneous polynomials $G_0,\ldots,G_k$ in the  variables $z_0,\ldots,z_k$ of the same degree $d$
    with no nontrivial  common factor  such that $$f= \lbrack G_0:\ldots :G_k\rbrack$$ in homogeneous coordinates.
    The polynomial map $F:=(G_0,\ldots,G_k)$ is  said to be  {\it a lifting} of   $f$ in $\C^{k+1}.$ The number $\d(f):=d$ is called  {\it the
    algebraic degree} of  $f.$ Moreover, $f$ is said to
    be {\it dominant} if   its Jacobian
    determinant  does not vanish identically (in  any local chart).
    In this  work we always consider   dominant meromorphic self-maps $f$ of $\P^k$ with $k\geq 2.$
   For $n\in\N,$  $f^n$  denotes  $f\circ\cdots\circ f$ ($n$ times).  The  Fatou  set of $f$  is  the largest  open subset of $\P^k$ on which 
   $(f^n)_{n=1}^{\infty}$  forms  a normal family. The Julia  set of $f$ is, by definition, the complement of its Fatou set  in $\P^k.$

Recall    the following definition (see \cite{fs,fs2,si})
\renewcommand{\thethmspec}{Definition 1}\begin{thmspec}
A  meromorphic self-map $f:\P^k\longrightarrow \P^k$   is said to be {\rm algebraically stable} (or  {\rm AS} for short) if
$\d(f^n)=\d(f)^n,$ $n\in\N.$
\end{thmspec}
In other words, $f$ is AS if and only if  a sequence  $(F_n)_{n=1}^{\infty}$  of liftings  of $(f^n)_{n=1}^{\infty}$
can be defined as follows
\begin{equation} \label{eq_AS}
F_n:=F_1\circ F_{n-1},\qquad n\geq 1,
\end{equation}
where $F_1,\ F_0$ are arbitrarily fixed  liftings of $f,\ f^0:=\Id$ respectively.

 For every AS  map $f$ with $\d(f) > 1,$ N. Sibony proves  in \cite{si} that
 the following limit in the sense of current
\begin{equation}\label{eq_Sibony}
T:=\lim_{n\to\infty} \frac{(f^n)^{*} \omega}{\d(f^n)}
\end{equation}
exists,
where  $\omega$ denotes the Fubiny-Study K\"{a}hler form on $\P^k$ so normalized that
 $\int\limits_{\P^k}\omega^k=1.$   $T$ is called {\it the Green current} associated to $f.$
 He also proves that  $T$ does not charge  any hypersurfaces.
Given a positive integer $d,$     a   ``generic"  meromorphic self-map  of  algebraic  degree $d$  is always  AS (see \cite{fs2}).
In the last  decades the study of  Green currents
 plays a central role in Complex Dynamics in higher dimensions.
  We address the reader to the survey articles of
  Forn{\ae}ss-Sibony  \cite{fs,fs2},  N. Sibony \cite{si}, Dinh-Sibony \cite{ds5}
   for    further explanations. Some other   articles on the topic  are       \cite{ds1,ds3, dds,gu} etc.

 In contrast to the case of  AS maps,  the dynamics  of  non  AS maps  are, at  the moment, still very poorly understood  although 
 there has been
a lot of activity around this topic in the past few years.
 Two fundamental   problems arise:

\smallskip

\noindent{\bf Problem 1.}  Study the degree-growth of non AS maps.

\noindent{\bf Problem 2.}
   Define  a  natural  Green currents  for  such  maps.

 \smallskip

  One of the first works
 in this direction is the article of Bonifant--Forn{\ae}ss \cite{bf}   where    some special non AS
maps are thoroughly studied.  In her thesis   \cite{bo} A.M. Bonifant constructs an appropriate Green current for these maps
and then writes down the functional equation.
 J. Diller and Ch. Favre (see  \cite{df}) have constructed Green currents for   birational maps of compact K\"{a}hler surface.
In the case of polynomial maps of  $\C^2,$ the two problems above have  been solved  by   Ch. Favre and  M. Jonsson (see \cite{fj,fj2}). Moreover, S. Boucksom,  Ch. Favre and  M. Jonsson   have investigated the
 degree-growth for meromorphic surface maps   (see \cite{bfj}).
In higher dimension  the situation is completely open, and the two questions above remain a great mystery.
B. Hasselblatt  and J.  Propp  have  studied  the degree-growth of monomial maps  in \cite{hp}.
There are series of interesting examples of birational maps acting on the space of complex square matrices which were worked out  by E. Bedford and K. Kim and others (see \cite{bk1,bk2} etc). 
However,
a general theoretical approach  still  does not exist at the moment.

\smallskip

In this paper we  study  a new class of non AS self-maps of $\P^k.$  Under  a  mild extra assumption on such maps  we construct the good Green currents  for them,
write  down the functional equations  and  show that the support of  the Green current  is  contained in the Julia set of the corresponding self-map.
Here  is the formal definition of the new   class.
\renewcommand{\thethmspec}{Definition 2}\begin{thmspec}
A  meromorphic self-map $f$ of $\P^k$   is said to be {\rm quasi-algebraically stable} (or {\rm  QAS} for short) if
a sequence of liftings $(F_n)_{n=1}^{\infty}$  for the iterates $(f^n)_{n=1}^{\infty}$ can be defined as follows
  \begin{equation*}
F_n:=
\begin{cases}
F_1\circ F_{n-1}, 
  &n=1,\ldots,n_0  ,\\
 \frac{F_1\circ F_{n-1}}{H\circ F_{n-n_0-1}} , 
&  n>n_0.
\end{cases}
\end{equation*}
 Here $n_0\geq 1$ is an integer, $H$ is a   homogeneous polynomial of $k+1$ variables,
and $F_0:=\Id$  and $F_1$  is an  arbitrarily fixed  lifting of  $f.$
\end{thmspec}

It is  worthy  comparing the above  recurrent formula  with (\ref{eq_AS}).
In the previous  work \cite{nv} the author  has introduced a criterion   in order to test  if a  non AS  self-map  is QAS  (see  Condition (i)--(iii) in Theorem  \ref{thm_in_nv} below).
In fact, in  \cite{nv} all self-maps which satisfy the latter criterion  were called  QAS. But in the present work  we   choose
Definition 2 as  the new definition  for QAS maps. Although this will enlarge  the class of QAS self-maps  we  emphasize that our main objective  is
 to construct canonical Green currents and to study their properties.
As it was  shown in  \cite{nv}   there  are a lot of non AS self-maps which are QAS. In the present  work
  a new family  of QAS  self-maps in $\P^2$ is  exhibited.

\smallskip

This paper is organized as follows.

\smallskip

We begin Section 2 by
collecting some background
and introducing some notation, in order to prepare for the statement of  the main theorem.

Section 3 is devoted to  the proof of  the main theorem.

 Finally, Section 4 concludes the paper with  a new family of  QAS self-maps in $\P^2.$

\medskip

\indent{\it{\bf Acknowledgment.}}
 The paper was written while  the  author was visiting  the Universit\'e  Pierre et Marie Curie and the Korea Institute  for Advanced Study (KIAS). He wishes to express his gratitude to these organizations. The  first version  of this article   circulated in a form of a preprint in   2002.

\medskip
\section{Statement of the main result} \label{section_main_results}
First  we fix  some notation and  terminology.
\subsection{Meromorphic self-maps and positive closed currents of bidegree $(1,1)$}\label{subsection_meromorphic_maps_currents_potientials}
 Let  $f$ be  a    meromorphic self-map of $\P^k.$  The   {\it indeterminacy locus} $\I(f)$ of $f$  is the set of all points of $\P^k$
    where $f$ is not  continuous, in other words, the common zero set
    of  component polynomials $G_0,\ldots,G_k,$  where  $(G_0,\ldots,G_k)$ is a lifting of $f.$  So $\I(f)$ is a subvariety of codimension at least $2.$
     The {\it first dynamical degree} of $f,$ denoted by $\lambda_1(f),$ is given by
    \begin{equation}\label{eq1.1}
    \lambda_1(f):=\lim\limits_{n\to\infty} \d(f^n)^{\frac{1}{n}}.
    \end{equation}
For a   discussion  on dynamical degrees of meromorphic self-maps see the articles of Dinh--Sibony \cite{ds2,ds4}.
 We denote by $\mathcal{C}^{+}(\P^k)$ the set of
 positive closed currents  of bidegree  $(1,1)$ on $\P^k.$ The mass of $T$  is defined by
  $\Vert T\Vert:= \int\limits_{\P^k} T\wedge \omega^{k-1}.$
 We consider the cone  $\mathcal{P}$ of plurisubharmonic  functions  $u$ in $\C^{k+1},$ satisfying the following homogeneity property:
there exist $c>0$ such that if $\lambda\in \C,$ then
\begin{equation*}
u(\lambda z)=c\log\vert \lambda\vert +u(z), \qquad  z\in \C^{k+1}.
\end{equation*}
The  functions in $\mathcal{P}$ are so normalized   that  $\sup\limits_{\B} u=0,$  where  $\B$  denotes the unit ball
in $\C^{k+1}.$

With a function $u$ satisfying  the above homogeneity property (but not necessarily the  normalization condition),  we are going to associate a current $T\in \mathcal{C}^{+}(\P^k).$
Let $\pi:\  \C^{k+1}\setminus \{0\}\rightarrow\P^k$ be the canonical projection. Let $U$ be an open set in $\P^k$ such that there is a
holomorphic inverse $s:\ U\rightarrow \C^{k+1}\setminus \{0\}$ of $\pi,$ that is, $\pi\circ s=\Id.$ Define $T$ on $U$ by  $T:=\ddc (u\circ s).$
Then $T$ is  independent of $s.$  With   this  local definition, we have  an operator  $\mathcal{L}$ defined by
$\mathcal{L}(u):=T.$  It is  well-known (see \cite{fs,si})  that  $\mathcal{L}$ is  an isomorphism between $\mathcal{P}$ and  $\mathcal{C}^{+}(\P^k).$ 
If $T=\mathcal{L}(u),$   then  we  say that  a function of the form  $u+c$ is  a potential  of $T,$  where $c\in\R$ is  a constant.  Given a potential  $u$  of   a current $T\in \mathcal{C}^{+}(\P^k),$   we have that
 $ u(\lambda z)=\Vert T\Vert\cdot \log\vert \lambda\vert +u(z)$ for $z\in \C^{k+1},\ \lambda\in\C.$ 

 Any (not necessarily reduced) hypersurface $\mathcal{H}$  of $\P^k$ defines a current of integration  $[\mathcal{H}]=[H=0]\in  \mathcal{C}^{+}_1(\P^k),$
 where $H:\ \C^{k+1}\longrightarrow\C$ is a homogeneous polynomial defining  $\mathcal{H}.$ Moreover,
 \begin{equation}\label{eq2.0}
 \Vert [H=0] \Vert =\deg(H),
\end{equation}
where $\deg(H)$ is the (homogeneous) degree of  $H.$

 For a current $T\in  \mathcal{C}^{+}(\P^k) ,$  we define  the pull-back of $T$ by $f$  as  follows
\begin{equation*}
f^{\ast}T:= \mathcal{L}(u\circ F),
\end{equation*}
where $u:=\mathcal{L}^{-1}(T)$ and  $F$ is  a lifting of $f.$
It is  easy to see that
 \begin{equation}\label{eq2.1}
 \Vert f^{\ast} T\Vert=\d(f)\cdot\Vert T\Vert.
 \end{equation}


 Finally, for a function $u:\ \C^{k+1}\rightarrow [-\infty,\infty],$  let $u^{\ast}$ denote its {\it upper semicontinuous regularization}, that is,
$u^\ast(z):=\limsup_{w\to z}u(w),$ $z\in \C^{k+1}.$

\subsection{Statement of the main result}\label{subsection_main_result}
Let $f$ be  a QAS self-map   as in Definition 2.  To simplify the notation, from now on we will write $\lambda$  for $\lambda_1(f)$ and $h$  for $\deg(H)$.
As an immediate consequence of Definition 2, we have that
\begin{equation}\label{eq1_section_preparatory_results}
\d(f^{n})=
 \begin{cases}
\d(f)^n, 
  &n=0,\ldots,n_0  ,\\
 \d(f)\cdot\d(f^{n-1})- h\cdot \d(f^{n-n_0-1}), 
&  n>n_0.
\end{cases}
\end{equation}
On the other hand, it has been shown in \cite[Theorem 4.2]{nv} that
$\lambda$ is the root of maximal modulus of the so-called {\it characteristic polynomial} of $f$
\begin{equation}\label{eq2_section_preparatory_results}
P(t)=t^{n_0+1}-dt^{n_0}+h.
\end{equation}

Now we are ready to formulate the following
\renewcommand{\thethmspec}{Main Theorem}\begin{thmspec}
    \label{mainthm}
   Under the above  hypothesis and notation, 
   suppose in addition that $\lambda$ is a  simple root of  $P(t)$ and  $\lambda>1.$
   \begin{itemize}
   \item[(i)]
  Then $\big(\limsup\limits_{n\to\infty}\frac{\log{\Vert F_n\Vert}}{\d(f^n)}\big)^{\ast}$ exists and defines
  a plurisubharmonic function $u$ in $\C^{k+1}.$
   \item[(ii)] Let  $T:=\mathcal{L}(u).$ Then the following functional equation holds
    \begin{equation*}
    f^{\ast}(T)=\lambda\cdot T+\frac{d-\lambda}{h}\cdot [H=0].
    \end{equation*}
    \item[(iii)] The support of the current $T$ is  contained in the Julia set of $f,$ which is  thus  non empty.
\end{itemize}
 \end{thmspec}

It is  worthy to remark here that the presence of a factor $[H=0]$  in the functional equation (ii) characterizes QAS self-maps. Indeed,
for  an AS map $f$ with corresponding Green current $T$  (see (\ref{eq_Sibony})),  the functional equation is  $f^{\ast}T=\d(f)\cdot T.$

We  know  from \cite[Theorem 4.2]{nv} that the multiplicity of $\lambda$ is  either $1$ or $2.$
It is  equal to $2$  only if     $h=\left(\frac{d}{n_0+1}\right)^{n_0+1}n_0^{n_0}.$ Moreover, it is  clear by (\ref{eq2_section_preparatory_results}) that
$\lambda>1$ if and only if $h>1.$

The  author does not know if the QAS self-maps  considered  in this work can be
made algebraically stable by blowing up the projective space, in which case the existence of
invariant currents is already known in general.  
However, the author thinks that  the idea of QAS self-maps and  the technique presented here  could
 be  extended  to  a more general context, for example,  meromorphic self-maps on compact
K\"{a}hler  manifolds etc.

\section{Proof}  \label{section_proof}
In this  section we keep the hypothesis and  notation introduced  just before  the Main Theorem
in Subsection  \ref{subsection_main_result}. 
 
\begin{lem}\label{lem_degree}
Let $r$ be  the  multiplicity of $\lambda$  in $P(t).$\\
1) There exist  $Q\in \R[t]$ with $\deg Q=r-1$ and  $0<\rho<1$  such that
\begin{equation*}
\d(f^n)=\lambda^n\big(Q(n) +\mathcal{O}(\rho^n)\big).
\end{equation*}
2) There  exists  a finite positive  constant $C$  such that for all $n\in\N,$
\begin{equation*}
\frac{\d(f^{n+1})-\lambda \d(f^n)}{\d(f^n)}\leq  \frac{C}{n}\qquad\text{and}\qquad
\sum\limits_{j=0}^n \d(f^j)\leq  C\d(f^n).
\end{equation*}
\end{lem}
\begin{proof} To prove Part 1) we proceed as in the proof of Theorem 4.2 in \cite{nv}.
Let $\lambda_1,\ldots,\lambda_{n_0+1}$ be all the roots of $P$  counted with multiplicities. There are two
cases to consider.

\smallskip

\noindent{\bf Case 1:} all roots of $P$ are distinct.

By the proof of Case  1 in Theorem 4.2 in \cite{nv}, there exist $c_j\in\C$ for $j=1,\ldots,n_0+1$ such that
\begin{equation}\label{eq1_lem_degree}
\d(f^n)=\sum\limits_{j=1}^{n_0+1}c_j\lambda^n_j,\qquad n\in\N.
\end{equation}
Moreover,  if 
\begin{equation*}
\vert
\lambda_1\vert=\max\limits_{c_j\not=0,\  j=1,\ldots, n_0+1}{\vert \lambda_j\vert}.
\end{equation*} 
then $\lambda_1>0$ and $c_1\in\R.$ This implies Part 1) for $\lambda:=\lambda_1$ and $Q(t):=c_1.$

\noindent{\bf Case 2:}  $P$ has a multiple root.

Recall from   Case 2  in the proof of Theorem 4.2 in \cite{nv} that
 $h=\left(\frac{d}{n_0+1}\right)^{n_0+1}n_0^{n_0}$ and that the only multiple root of $P$ is 
$\lambda_{n_0}:=\frac{dn_0}{n_0+1}$   which    is, in fact,  a double root.
We may assume  without loss of generality that $\lambda_{n_0+1}=\lambda_{n_0}.$
Moreover,  there exist $c_j\in\C$ for $j=1,\ldots,n_0-1,$ and $c_{n_0},c_{n_0+1}\in\R$  such that
\begin{equation}\label{eq2_lem_degree}
\d(f^n)=\sum\limits_{j=1}^{n_0-1} c_j\lambda^n_j+(nc_{n_0}+c_{n_0+1})\lambda_{n_0}^n,\qquad n\in\N.
\end{equation}
Let 
\begin{equation*} 
\mu:=\max\limits_{c_j\not=0,\  j=1,\ldots, n_0+1}{\vert \lambda_j\vert}.
\end{equation*} 
There are two subcases to consider.

\noindent{\bf  Case 2a:} $\mu\not=\lambda_{n_0}.$ 

Then $\mu >\lambda_{n_0}.$   We  argue as in Case 1 using  (\ref{eq2_lem_degree}) instead of (\ref{eq1_lem_degree}). Consequently, if $|\lambda_1|=\mu$ then we can show that $\lambda_1>0$ and $c_1\in\R.$
This  proves Part 1) for $\lambda:=\lambda_1$ and $Q(t):=c_1.$

\noindent{\bf  Case 2b:} $\mu=\lambda_{n_0}.$

 By  (\ref{eq2_lem_degree})  we  see that Part 1) holds for $\lambda:=\lambda_{n_0}$ and $Q(t):=c_{n_0}t+c_{n_0+1}.$
 This completes the proof of Part 1).

Part 2)  is an immediate consequence  of Part 1).
\end{proof}

In this  section we make the following convention: $\d(f^{n})=0$ for all $n< 0.$

\begin{lem}\label{prop_recurrence}
The following identity holds
\begin{equation}\label{eq1_prop_recurrence}
F_{n}=\begin{cases}
F_{n-1}\circ F, 
  &n=1,\ldots,n_0  ,\\
  \frac{F_{n-1}\circ F}{H^{\d(f^{n-n_0-1})}},&  n>n_0.
\end{cases}
\end{equation}
Moreover for all currents  $T\in\mathcal{C}^{+}(\P^k),$
 \begin{equation}\label{eq2_prop_recurrence}
  (f^n)^{*}T=
  \begin{cases}
f^{\ast}((f^{n-1})^{*}T), 
  &n=1,\ldots,n_0  ,\\
 f^{\ast}((f^{n-1})^{*}T)-\Vert T\Vert\cdot \d(f^{n-n_0-1})\cdot [H=0], 
&  n>n_0.
\end{cases}
 \end{equation}
\end{lem}


\begin{proof}
We only need to prove identity  (\ref{eq1_prop_recurrence}) since identity (\ref{eq2_prop_recurrence})
is an immediate consequence of (\ref{eq1_prop_recurrence}) using (\ref{eq2.0})--(\ref{eq2.1}).
Observe that by the hypothesis on $f,$   (\ref{eq1_prop_recurrence}) is true for $n=1,\ldots,n_0.$
Supposing  (\ref{eq1_prop_recurrence}) true for $n,$ we need to prove it for $n+1.$

We have that
\begin{eqnarray*}
F_{n}\circ F&=&\frac{F\circ F_{n-1}\circ F}{H\circ F_{n-n_0-1}\circ F}=\frac{F( F_{n-1}\circ F)}{H( F_{n-n_0-1}\circ F)}\\
&=&\frac{F(H^{\d(f^{n-n_0-1})} F_n)}{H( H^{\d(f^{n-2n_0-1})} \cdot F_{n-n_0})}
=\frac{H^{\d(f)\cdot \d(f^{n-n_0-1})}\cdot F\circ F_n}{       H^{h\cdot\d(f^{n-2n_0-1})} \cdot H\circ F_{n-n_0}}\\
&=&   H^{\d(f)\cdot \d(f^{n-n_0-1})-h\cdot\d(f^{n-2n_0-1})}    F_{n+1}= H^{\d(f^{n-n_0})}    F_{n+1},
\end{eqnarray*}
where the first equality follows  from Definition 2, the third one from the hypothesis of induction,
and the last one from identity (\ref{eq1_section_preparatory_results}).
Hence,  (\ref{eq1_prop_recurrence}) is true for $n+1.$
\end{proof}

Put, for $N\geq 1,$
\begin{equation*}
\sigma_N:=\frac{1}{N}\sum\limits_{n=0}^{N-1} \frac{1}{\d(f^n)} (f^n)^{\ast}\omega.
\end{equation*}
Then  $(\sigma_N)$ is a sequence of positive closed currents of bidegree $(1,1)$ such that  $\Vert \sigma_N\Vert=1.$
Therefore, we can extract a  convergent subsequence $(\sigma_{N_j}):$    $\sigma_{N_j}\to \sigma.$ Here,
$\sigma$ is a positive closed currents of bidegree $(1,1)$ such that  $\Vert \sigma\Vert=1.$


\begin{lem}\label{prop_functional_equation_sigma}
 The following functional equation holds
    \begin{equation*}
    f^{\ast}\sigma=\lambda\cdot \sigma+\frac{d-\lambda}{h}\cdot [H].
    \end{equation*}
\end{lem}

\begin{proof}
We have that
\begin{eqnarray*}
f^{\ast}\sigma_N-\lambda\sigma_N   &=& \frac{1}{N}\sum\limits_{n=0}^{N-1} \left( \frac{f^{\ast}((f^{n})^{\ast}\omega)}{\d(f^{n})}-
\frac{\lambda (f^{n+1})^{\ast}\omega  }{\d(f^{n+1})}
 \right) +\frac{1}{N}\Big( \lambda \frac{(f^N)^{\ast} \omega}{\d(f^N)}  -\lambda\omega\Big)\\
&=&  \frac{1}{N}\sum\limits_{n=0}^{N-1} \frac{f^{\ast}(f^{n})^{\ast}\omega- (f^{n+1})^{\ast}\omega}{\d(f^{n})}
   \\
 &+&    \Big(    \frac{1}{N}\sum\limits_{n=0}^{N-1} \frac{\d(f^{n+1})-\lambda\d(f^n)}{\d(f^{n})}
 \cdot\frac{(f^{n+1})^{\ast}\omega}{\d(f^{n+1})}
 +\frac{\lambda}{N}\left( \frac{(f^N)^{\ast} \omega}{\d(f^N)}  -\omega \right)\Big)\\
& :=&  I+II.
\end{eqnarray*}
Applying Lemma \ref{prop_recurrence} yields that
\begin{equation}\label{eq1_prop_functional_equation_sigma}
I=\Big( \frac{1}{N}\sum\limits_{n=0}^{N-1} \frac{\d(f^{n-n_0})}{\d(f^{n})}\Big)  [H].
\end{equation}
 Recall from  (\ref{eq2.1})  that $ \Vert \frac{(f^{n})^{\ast}\omega}{\d(f^{n})}\Vert=1, $ $n\geq 0.$ Therefore,
 \begin{equation*}
 \frac{\lambda}{N}\left( \Big\Vert \frac{(f^N)^{\ast} \omega}{\d(f^N)} \Big\Vert +\Vert\omega\Vert \right)\to 0
 \quad \text{as}\ N\to\infty.
 \end{equation*}
 On the other hand, applying  the first estimate of Part 2) of  Lemma \ref{lem_degree}  yields that
\begin{equation*}
\frac{1}{N}\sum\limits_{n=0}^{N-1} \frac{\vert \d(f^{n+1})-\lambda\d(f^n)\vert}{\d(f^{n})}
 \cdot \Big\Vert\frac{(f^{n+1})^{\ast}\omega}{\d(f^{n+1})}\Big\Vert\leq  \frac{C}{N}\sum\limits_{n=0}^{N-1}
  \frac{1}{n}\leq  {C \log N\over N} \to 0,\quad  \text{as}\ N\to\infty.
\end{equation*}
Inserting the last two estimates into the expression of $(II),$ we obtain that $II\to 0$ as $N\to\infty.$
This, combined with (\ref{eq1_prop_functional_equation_sigma}) implies that
 \begin{equation*}
    f^{\ast}\sigma=\lambda\cdot \sigma+\mu [H]
    \end{equation*}
    for some $\mu\in\R.$ By equating the mass of both sides in the last equation
    and  using  (\ref{eq2.1}),   the lemma follows.
    \end{proof}

    By Lemma \ref{prop_functional_equation_sigma}, we can fix a potential\footnote{ See Subsection \ref{subsection_meromorphic_maps_currents_potientials}.} $\Theta$ of $\sigma$ such that
    \begin{equation}\label{eq_for_potential_Theta}
    \Theta\circ F=\lambda\Theta+ \frac{d-\lambda}{h}\log{ \vert H\vert}.
    \end{equation}


\begin{lem}\label{prop_recurrence_formula_for_Theta}
 \begin{equation*}
  \Theta\circ F_n=
 \begin{cases}
\lambda^n\Theta+\frac{d-\lambda}{h}\cdot\sum\limits_{j=1}^n \lambda^{j-1}\log{\vert H\circ F_{n-j}\vert},
  &n=1,\ldots,n_0  ,\\
\lambda^n\Theta+\frac{d-\lambda}{h}\cdot\sum\limits_{j=1}^{n_0} \lambda^{j-1}\log{\vert H\circ F_{n-j}\vert} ,
&  n>n_0.
\end{cases}
 \end{equation*}
\end{lem}


\begin{proof}
We proceed by induction. For $n=1$ the above  formula follows from  (\ref{eq_for_potential_Theta}).

Suppose that the above  inductive formula is true for $n.$ We need to show it for $n+1.$
Observe that
\begin{eqnarray*}
\Theta\circ F_{n+1}&=&\Theta\circ F_n\circ F-\d(f^{n-n_0})\log \vert H\vert\\
&=&\lambda^n\Theta\circ F+\frac{d-\lambda}{h}\cdot\sum\limits_{j=1}^{n_0} \lambda^{j-1}\log{\vert H\circ F_{n-j}\circ F\vert}
-\d(f^{n-n_0})\log \vert H\vert\\
&=&\lambda^{n+1}\Theta+\frac{d-\lambda}{h}\cdot \lambda^n\log{\vert H\vert}-\d(f^{n-n_0})\log \vert H\vert\\
&+&\frac{d-\lambda}{h}\cdot\sum\limits_{j=1}^{n_0} \lambda^{j-1}\log{\vert H(H^{\d(f^{n-j-n_0})}\cdot F_{n-j+1})\vert}\\
&=&\lambda^{n+1}\Theta+\frac{d-\lambda}{h}\cdot\sum\limits_{j=1}^{n_0} \lambda^{j-1}\log{\vert H\circ F_{n-j+1}\vert}\\
&+&\Big (\frac{d-\lambda}{h}\cdot \lambda^n
+\frac{d-\lambda}{h}\cdot\sum\limits_{j=1}^{n_0} \lambda^{j-1}h\d(f^{n-j-n_0})  -\d(f^{n-n_0})      \Big)   \log{\vert H\vert}\\
&=& \lambda^{n+1}\Theta+\frac{d-\lambda}{h}\cdot\sum\limits_{j=1}^{n_0} \lambda^{j-1}\log{\vert H\circ F_{n-j+1}\vert}\\
&+&\Big ( \lambda^{n-n_0}
+\frac{d-\lambda}{h}\cdot\sum\limits_{j=1}^{n_0} \lambda^{j-1}h\d(f^{n-j-n_0})  -\d(f^{n-n_0})      \Big)   \log{\vert H\vert},
\end{eqnarray*}
where the first equality follows from (\ref{eq1_prop_recurrence}), the second one from the hypothesis of induction,
the third one from  (\ref{eq_for_potential_Theta}) and  (\ref{eq1_prop_recurrence}), and the last one
from (\ref{eq2_section_preparatory_results}).
Therefore,   the proof of the inductive formula will be complete for $n+1$ if one can show that for all $n\geq 0,$ $S_n=0,$ where
\begin{equation*}
 S_n:=\lambda^{n}
+(d-\lambda)\cdot\sum\limits_{j=1}^{n_0} \lambda^{j-1}\d(f^{n-j})  -\d(f^{n})     .
\end{equation*}
It follows from (\ref{eq1_section_preparatory_results})--(\ref{eq2_section_preparatory_results}) and the above formula for $S_n$ that
$S_{n}-dS_{n-1}+hS_{n-n_0-1}=0$ for all $n\geq n_0+1.$ Hence, the proof will be complete if one can show that
$S_n=0$ for $n=0,\ldots,n_0.$ But the last assertion is equivalent to the identity
 \begin{equation*}
 \lambda^{n}
+(d-\lambda)\cdot\sum\limits_{j=1}^{n_0} \lambda^{j-1} \d(f)^{n-j}  =\d(f)^n     ,
\end{equation*}
which is  clearly  true by using the  convention preceding Lemma   \ref{prop_recurrence}. Hence, the proof is complete.
\end{proof}

The following elementary lemma is needed.
\begin{lem}\label{lem_elementary} Let $(X,\mu)$ be a measurable space  and
$(g_n)_{n=0}^{\infty},\ (h_n)_{n=1}^{\infty}\subset L^1(X,\mu)$ two sequence of complex-valued functions with  $\Vert  h_n\Vert_{L^1(X)}\leq 1,$ $n\geq 1.$  Let  $P(t):=t^{n_0}+\alpha_1t^{n_0-1}+\cdots+\alpha_{n_0}$ be a polynomial whose roots are of modulus  strictly smaller than $1.$
Let $(\epsilon_n)_{n=n_0}^{\infty}\subset \R^{+}$ be a sequence  with $\lim\limits_{n\to\infty} \epsilon_n=0.$
Let  $(\alpha_{1n})_{n_0}^{\infty},\ldots,(\alpha_{n_0n})_{n_0}^{\infty}\subset \C$  be $n_0$ sequences     such that
for all $n\geq n_0$  and $1\leq  j\leq n_0,$
\begin{itemize}
\item[$\bullet$] $g_n+\alpha_{1n} g_{n-1}+\cdots+\alpha_{n_0n}g_{n-n_0}=h_n;$
    \item[$\bullet$]  $\vert  \alpha_{jn}-\alpha_j\vert<\epsilon_n. $
    \end{itemize}
 Then $(g_n)_{n=0}^{\infty}$ is  bounded in  $L^1(X,\mu)$.
\end{lem}
\begin{proof}
 Let  $t_1,\ldots,t_{n_0}$  be  the roots of
$P(t).$ Consider  two cases.

\noindent {\bf Case 1:}
$t_1,\ldots,t_{n_0}$ are distinct.

We can  check that  if $\gamma_1,\ldots,\gamma_{n_0}\in\C$  such that
 $\sum_{j=1}^{n_0} \gamma_j\cdot \frac{P(t)}{t-t_j}\equiv 0$
  then  $\gamma_1=\cdots=\gamma_{n_0}=0.$ Consequently, there  exist $\gamma_1,\ldots,\gamma_{n_0}\in\C$ such that
 \begin{equation}\label{eq0_lem_elementary}
 \sum_{j=1}^{n_0}\gamma_j\cdot \frac{P(t)}{t-t_j}\equiv t^{n_0-1}.
 \end{equation}
Next, write
\begin{equation}\label{eq1_lem_elementary}
\frac{P(t)}{t-t_j}=t^{n_0-1}+\beta_{j1}t^{n_0-2}+\cdots +\beta_{jn_0-1},\qquad  j=1,\ldots,n_0 .
\end{equation}
Put
\begin{equation*}
f_{jn}:=g_n+\beta_{j1}g_{n-1}+\cdots + \beta_{jn_0-1} g_{n-n_0+1},\qquad  n\geq n_0-1,\ 1\leq j\leq n_0.
\end{equation*}
Hence, (\ref{eq0_lem_elementary}) becomes
\begin{equation}\label{eq2_lem_elementary}
 \sum_{j=1}^{n_0}\gamma_j f_{jn}= g_{n-1},\qquad   n\geq n_0-1,\ 1\leq j\leq n_0.
 \end{equation}
The formula  for $f_{jn},$  the first  $\bullet$   of the hypothesis and  identity
(\ref{eq1_lem_elementary}) together  imply that
\begin{equation*}
\vert f_{jn}-t_j   f_{j,n-1}\vert\leq \vert h_n\vert +
 \epsilon_{n-1}\vert g_{n-1}\vert+\cdots+\epsilon_{n-n_0}\vert g_{n-n_0}\vert,\qquad  n\geq n_0,\  1\leq j\leq n_0.
\end{equation*}
 It follows from the last estimate
  and (\ref{eq2_lem_elementary}) that there  is   a  finite positive  constant $C$
 such that
 \begin{equation*}
 M^{'}_n\leq \rho^{'} M^{'}_{n-1}+  C (\epsilon_{n-1}+\cdots+\epsilon_{n-n_0})(M^{'}_n+\cdots+M^{'}_{n-n_0+1}) +\Vert h_n\Vert_{L^1(X)}, \qquad  n \geq 2n_0,
 \end{equation*}
 where  $M^{'}_n:=\max\{\Vert f_{1n}\Vert_{L^1(X)},\ldots,\Vert f_{n_0n}\Vert_{L^1(X)} \}$ for all $n\geq n_0-1,$
 and  $\rho^{'}:=\max\limits_{1\leq j\leq n_0} \vert t_j\vert.$ Observe that  $0<\rho^{'}<1$ since
 $\vert t_j\vert <1.$

Fix a constant $\rho:$  $\rho^{'}<\rho<1.$
 Using the above estimate for $M^{'}_n$  repeatedly and taking into account that  $\lim\limits_{n\to\infty}\epsilon_n=0,$
 we may  find     a  sufficiently large integer $N>n_0$ such that
 \begin{equation*}
 M_{n}\leq \rho M_{n-1}+  2 \big(\sum\limits_{j=1}^{n_0} \Vert h_{nn_0-j}\Vert_{L^1(X)}\big), \qquad  n > N,
 \end{equation*}
 where  $M_n:=\max\{\ M^{'}_{nn_0-1},\ldots, M^{'}_{nn_0-n_0} \}$ for all $n\geq N .$
 Consequently,
\begin{equation*}
M_n\leq \frac{2}{1-\rho}\cdot\sum\limits_{k=0}^{n-N-1}  \rho^k  \big (\sum\limits_{j=1}^{n_0} \Vert h_{(n-k)n_0-j}\Vert_{L^1(X)}\big) +\rho^{n-N}M_N
\end{equation*}
 for all $n\geq N.$ This, combined with the hypothesis that  $\Vert  h_n\Vert_{L^1(X)}\leq 1$ for all $n\geq 1,$
 implies the existence of  a finite positive  constant $M$ such that
\begin{equation*}
\big\Vert f_{jn}\big\Vert_{L^1(X)}  <M,\qquad  n\geq N,\  1\leq j\leq n_0.
\end{equation*}
This, coupled with (\ref{eq2_lem_elementary}), gives  the desired  conclusion.

 \noindent {\bf Case 2:}
$t_1,\ldots,t_{n_0}$ are not distinct.

 Let $t_1,\ldots,t_r$  be  all distinct  roots of $P(t)$  with multiplicity
$m_1,\ldots,m_r$ respectively.
We  can  choose  $\gamma_{11},\ldots,\gamma_{1m_1},\ldots,\gamma_{r1},\ldots,\gamma_{rm_r}\in\C$ such that
 \begin{equation*}
 \sum_{j=1}^{r} \frac{(\gamma_{j1}+\gamma_{j2}t+\cdots+\gamma_{jm_j}t^{m_j-1}            )P(t)}{(t-t_j)^{m_j}}\equiv t^{n_0-1}.
 \end{equation*}
 The   remaining part of the proof follows  along the same lines as  in the previous case.
\end{proof}

Let us recall that a {\it quasi-plurisubharmonic function on $\P^k$} is an upper semi-continuous function $\phi:\ \P^k\rightarrow [-\infty,\infty)$
which is locally given as  the sum of a plurisubharmonic and a smooth function. The following estimate due to V.  Guedj (see Proposition 1.3 in \cite{gu})   is needed.
\begin{lem}\label{lem_Guedj}
There exists a positive finite constant $C$  such that for all quasi-plurisubharmonic functions $\phi$ with  $\max\phi=0,$ $\ddc\phi\geq -\omega,$
and for all $n\in\N,$
\begin{equation*}
\int\limits_{\P^k}(\vert \phi\vert\circ f^n)\omega^k\leq C\sum\limits_{j=0}^n \d(f^j).
\end{equation*}
\end{lem}

Now  we arrive at

\smallskip

\noindent{\bf Proof of Part (i) of Main Theorem.}
For all $n>n_0$ consider the functions defined on $\P^k$
\begin{equation*}
h_n:=\Theta+\frac{d-\lambda}{h}\cdot\sum\limits_{j=1}^{n_0} \frac{1}{\lambda^{n-j+1}}\log{\left\vert H\big(\frac{ F_{n-j}}{\Vert F_{n-j}\Vert}\big)\right \vert}-\frac{1}{\lambda^n}\Theta \big(\frac{ F_{n}}{\Vert F_{n}\Vert}\big).
\end{equation*}
By Lemma  \ref{prop_recurrence_formula_for_Theta}, we  have that
\begin{equation}\label{eq1_application_lem_elementary}
\frac{\log{\Vert F_n\Vert}}{\lambda^n}-\frac{d-\lambda}{\lambda}\cdot\sum\limits_{j=1}^{n_0} \frac{\log{\Vert F_{n-j}\Vert}}{\lambda^{n-j}} =h_n .
\end{equation}
On the one hand, since $\lambda$ is  a  simple  root of $P$ we know  from Part 1) of  Lemma  \ref{lem_degree} that  $\d(f^n)\approx \lambda^n.$
On the other hand, since  $\Theta$ and  $\log{\vert H\vert}$  are plurisubharmonic in $\C^{k+1},$ an application of Lemma \ref{lem_Guedj}
and the second estimate of Part 2) of  Lemma \ref{lem_degree}  gives that  $\Vert h_n\Vert_{L(\P^k,\omega^k)}<C$ for a finite constant $C$ independent of $n.$
Moreover,  the polynomial  $t^{n_0} -\frac{d-\lambda}{\lambda}\cdot\sum\limits_{j=1}^{n_0} t^{n_0-j}$ is  equal to
$\frac{1}{\lambda^{n_0}} \frac{P(\lambda t)}{\lambda t-\lambda}$ by  using the identity
$P(\lambda t)=P(\lambda t)-P(\lambda).$ Therefore, by (\ref{eq2_section_preparatory_results}) 
and  by the hypothesis that $\lambda$ is a  simple root of $P,$  we see that  all roots  of    $t^{n_0} -\frac{d-\lambda}{\lambda}\cdot\sum\limits_{j=1}^{n_0} t^{n_0-j}$
are of modulus  strictly smaller than $1.$
Hence, we are in the position to apply  Lemma \ref{lem_elementary} to  the relations (\ref{eq1_application_lem_elementary}) with $\alpha_j:=-\frac{d-\lambda}{\lambda}$ and   $\alpha_{jn}:=-\frac{d-\lambda}{\lambda},$  $1\leq  j\leq n_0.$
Consequently, it follows  that  $  \frac{\log{\Vert F_n\Vert}}{\lambda^n}$ is locally  uniformly bounded in $L^1(\C^{k+1})$-norm.
This proves Part (i).

\smallskip

\noindent{\bf Proof of Part (ii).} Using identity (\ref{eq1_prop_recurrence}) we have
\begin{eqnarray*}
\frac{\log{\Vert F_n\circ F\Vert}}{\d(f^n)}&=&\frac{\log{\Vert F_{n+1}\Vert}}{\d(f^n)}+
\frac{\d(f^{n-n_0})\log{\vert H\vert}}{\d(f^n)}\\
&=&\lambda\cdot\frac{\log{\Vert F_{n+1}\Vert}}{\d(f^{n+1})}+\Big( \frac{\d(f^{n+1})}{\d(f^n)}-\lambda \Big)
\frac{\log{\Vert F_{n+1}\Vert}}{\d(f^{n+1})}+ \frac{\d(f^{n-n_0})\log{\vert H\vert}}{\d(f^n)}.
\end{eqnarray*}
 Now take the $(\limsup\limits_{n\to\infty})^{\ast}$ of both sides of the above identity.
  By Part (i), the left hand side  is then $u\circ F$ and the first term of the right hand side is $\lambda\cdot u.$
  The second term of the right hand side   is  $0$  by using the first estimate of Part 2) of Lemma \ref{lem_degree}
  and  the fact already proved in Part (i)  that
  $  \frac{\log{\Vert F_n\Vert}}{\d(f^n)}$ is locally  uniformly bounded in $L^1(\C^{k+1})$-norm.
  The last term of the right hand side  converges to $\frac{1}{\lambda^{n_0}}\cdot\log{\vert H\vert}$
  using Part 1) of Lemma \ref{lem_degree}: $\d(f^n)\approx \lambda^n.$
Summarizing, we have shown that
\begin{equation*}
u\circ F=\lambda\cdot u+\frac{1}{\lambda^{n_0}}\cdot\log{\vert H\vert}=\lambda\cdot u+\frac{d-\lambda}{h}\cdot\log{\vert H\vert},
\end{equation*}
where the last equality follows from equation (\ref{eq2_section_preparatory_results}).
This proves (ii).

\smallskip

\noindent{\bf Proof of Part (iii).}
Let $p\in U,$ where $U$ is  an open set contained in the Fatou set. Shrinking $U$ if  necessary, we may
assume  that  a  subsequence ${f^{n_j}}$ converges  in $U$ to  a holomorphic map $h$ and that
$f^{n_j}(U)\subset \{z_0=1, \vert z_j\vert <2\}.$  We can then  write
\begin{equation*}
\frac{\log\Vert F_{n_j}\Vert}{ \d(f^{n_j})}=
\frac{\log\Vert (F_{n_j})_0\Vert}{ \d(f^{n_j})}+\frac{1}{ \d(f^{n_j})}\log \Vert(1, A^1_j,\ldots,A^k_j  \Vert.
\end{equation*}
The last term converges  uniformly to $0,$  and the first term is pluriharmonic. Hence, using Part (i)
the function $u$ is pluriharmonic on $U,$ and $U$ does not intersect the support of $T.$


%
%
%
%
%

%
%
%
%
%
 \section{Examples} \label{section_Example}
First we recall the result  from our previous  work  \cite{nv}.
\subsection{A sufficient condition  for QAS self-maps}
In \cite{fs2} Forn{\ae}ss and Sibony give the following definition.
\begin{defi}\label{degreelowering} {\rm A  hypersurface $\mathcal{H}\subset \P^k$
is said to be a {\it degree lowering  hypersurface} of $f$ if, for some (smallest) $n\geq
1,$   $f^n(\mathcal{H})\subset \I(f).$ The integer $n$ is then called the {\it  height} of $\mathcal{H}.$}
\end{defi}

The following (see Proposition 3.2 in \cite{nv}) gives us the structure of a non AS self-map.
 \begin{prop}\label{structure}
    Let   $f $ be a   meromorphic self-map of $\P^k.$   Then there are exactly an
    integer $M\geq 0,$  $M$ degree lowering  hypersurfaces $\mathcal{H}_j$   with height
    $n_j,$ $j=1,\ldots,M,$   satisfying the following properties:
     \\
    (i)  All the numbers $n_j$ $,j=1,\ldots,M,$ are  distinct.\\
    (ii)  $\codim\left(f^m(\mathcal{H}_j)\right) >1$ for
    $ m=1,\ldots,n_j,$ and $j=1,\ldots,M.$\\
    (iii) For any   degree lowering  irreducible hypersurface $\mathcal{H}$ of $f,$ there are
     integers $n\geq 0$ and  $1\leq j\leq M$ such that $f^n(\mathcal{H})$ is a  hypersurface and
    $f^n(\mathcal{H})\subset  \mathcal{H}_j$
 \end{prop}
 So    $f$ is AS if and only if  $M=0.$

\begin{defi} \label{primitive} {\rm Under the hypothesis and  the notation of  Proposition \ref{structure}, for every $ j=1,\ldots,M,$
 $\mathcal{H}_j$ is said to be  {\it the primitive   degree lowering  hypersurface} of
 $f$  with the {\it height} $n_j.$} 
\end{defi}

We are now able to  state  a  sufficient criterion for QAS maps  (see Main Theorem in \cite{nv}).
\begin{thm} \label{thm_in_nv}
 A meromorphic self-map $f $ of $\P^k$ is QAS
if it satisfies the following properties  (i)--(iii):
 \begin{itemize}
\item[$(i)$]   There is only one primitive degree lowering
hypersurface, in other words, $M=1.$
\item[$(ii)$]   Let  $\mathcal{H}_0$ be the hypersurface from Part (i) and let  $n_0$ be its height. Then for every   irreducible component $\mathcal{H}$
of $\mathcal{H}_0$ and every $m=1,\ldots,n_0,$  $f^{m}(\mathcal{H})\not\subset\mathcal{H}_0;$
  \item[$(iii)$] For every   irreducible component $\mathcal{H}$
of $\mathcal{H}_0,$   one of the following two conditions holds\\
    $(iii)_1$  $f^{m}(\mathcal{H})\not\subset \I(f)$ for all $m\geq n_0+1,$\\
     $(iii)_2$    there is an $m_0\geq n_0$ such that
 $f^{m_0+1}(\mathcal{H})$ is a hypersurface and $f^{m}(\mathcal{H})\not\subset \I(f)$ for all $m$ verifying $n_0+1\leq m\leq m_0.$
\end{itemize}
\end{thm}
Theorem \ref{thm_in_nv} coupled  with Proposition \ref{structure} gives us an efficient and simple method  to check if a map  is QAS.  The remaining  of this  section is devoted  to
the study of
  new parameterized families
of QAS maps in $\P^2.$
\subsection{New family of QAS self-maps of $\P^2$}
Let  $P$  be  a (not  necessarily irreducible)  homogeneous  polynomial  in $\C^3.$
Let $Q_1,Q_2$ and $ Q_3$    be  (not  necessarily irreducible)  homogeneous  polynomials  in $\C^3$ of the same degree.
Let  $R$  be   a (not  necessarily irreducible)  homogeneous  polynomial  in  $\C^3$  such that
$\deg(R)=\deg(P)+\deg(Q_1)$  and  that
\begin{equation}\label{eq1_example}
P(1,1,1)Q_j(1,1,1)=R(1,1,1)\not=0,\qquad   j=1,2,3.
\end{equation}
Suppose for the moment that $PQ_1-R,\ PQ_2-R,\ PQ_3-R$ have  no nontrivial common factor,
we are able to define  a dominant meromorphic  map  of $\P^2$
\begin{equation}\label{eq2_example}
f([z:w:t]):=\left\lbrack PQ_1-R:PQ_2-R:PQ_3-R             \right\rbrack.
\end{equation}
It can be checked  that for every  $(a,b,c)\in\C^3\setminus\{0\}$ with $a+b+c=0,$
the hypersurface $\{aQ_1+bQ_2+cQ_3=0\}$ is  sent  by $f$ into  the complex line $\{az+bw+ct=0\}.$
\begin{prop} \label{prop_new_family}
 Suppose  that for all  $(a,b,c)\in\C^3\setminus\{0\}$ with $a+b+c=0,$
every  irreducible component of the hypersurface $\{aQ_1+bQ_2+cQ_3\}$ is  sent  by $f$ onto a hypersurface.
Suppose  in addition that every  irreducible component of the hypersurface $\{P=0\}$ is  sent  by $f^2$ onto a hypersurface  and that  if $\mathcal{G}$  is  an irreducible hypersurface such that $f(\mathcal{G})=[1:1:1]$
then $\mathcal{G}\subset \{P=0\}.$
 Then $f$  satisfies the properties (i)--(iii) listed in Theorem \ref{thm_in_nv},
in particular,  $f$  is  QAS.
\end{prop}
\begin{proof}
First observe   by  (\ref{eq2_example})  and (\ref{eq1_example}) that  the hypersurface $\{ P=0\}$  is  sent  by $f$  to  the point $[1:1:1]\in \I(f).$
We  will  show that  there  is  no irreducible  degree  lowering hypersurface  other  than  the components of $  \{ P=0\}.$
To do this  suppose, in order  to get a  contradiction,  that
 $G$ is an irreducible homogeneous polynomial in $\C^3$ such that $\mathcal{G}\not\subset\{P=0\}$ and  that
  $f(\mathcal{G})$ is a point $[a:b:c]\in\P^2,$ where $\mathcal{G}$ is the hypersurface $\{G =0\}$ in $\P^2.$
  Suppose, without loss of generality, that  $a\not=0.$
  We deduce from (\ref{eq2_example}) and the equality  $f(\mathcal{G})=[a:b:c]$  that
  $G$ divides both polynomials
 $  P(bQ_1-aQ_2)-(b-a)R$ and
  $  P(cQ_1-aQ_3)-(c-a)R.$ Hence,
  $G$ divides  the polynomial
  \begin{equation*}
   P\cdot \Big( (c-a)(bQ_1-aQ_2)-(b-a)(cQ_1-aQ_3)\Big).
  \end{equation*}
  Since   $\mathcal{G}\not\subset\{P=0\}$ and  $G$ is  irreducible,  we  see that $G$  divides
  the polynomial  $(ac-ab)Q_1+(a^2-ac)Q_2+(ab-a^2)Q_3.$  Since $a\not=0,$    we deduce  from
  the first hypothesis  that either  $f(\mathcal{G})$ is  a hypersurface or $a=b=c.$ The former case  contradicts   the assumption that  $f(\mathcal{G})$ is a point $[a:b:c]\in\P^2.$ The latter case implies that
  $f(\mathcal{G})=[1:1:1],$ which, by the third hypothesis, gives that  $\mathcal{G}\subset\{P=0\},$  which
  contradicts our assumption.

 We have  shown that  $  \{ P=0\}$ is  the unique primitive  degree lowering hypersurface and its height is $1.$
Since by (\ref{eq1_example}) $\left\lbrack
1:1:1\right\rbrack\not\in  \{ P=0\},$  and  every  irreducible component of the hypersurface $\{P=0\}$ is  sent  by $f^2$ onto a hypersurface, it follows that     $f$   satisfies  (i)-(ii)-$\text{(iii)}_2$ of Theorem  \ref{thm_in_nv}.
\end{proof}

Now we will discuss  cases when  the  hypotheses  of  Proposition \ref{prop_new_family} are fulfilled.

\begin{cor}\label{cor_criterion1}
Suppose that $Q_2-Q_1,$   $Q_3-Q_1$ are coprime and that $P,$  $R$ are coprime. Then
$PQ_1-R,$ $PQ_2-R,$ $PQ_3-R$ have  no nontrivial common factor. Moreover,
for every
     irreducible hypersurface  $\mathcal{G}$ with $f(\mathcal{G})=[1:1:1],$
we have  $\mathcal{G}\subset \{P=0\}.$ Here  $f$ is  defined  by (\ref{eq2_example}).
\end{cor}
\begin{proof}
It is  left to the interested reader as an exercise.
\end{proof}

\begin{cor}\label{cor_criterion2} Suppose that the pre-image of the point $[1:1:1]$ by the map  $\P^2\ni [z:w:t]\mapsto [Q_1:Q_2:Q_3]$ is
a set of finite points and that
  for  every  $[z:w:t]\in \{P=0\}\cap \{R=0\}$ and every  $(a,b,c)\in\C^3\setminus\{0\}$ with $a+b+c=0,$ we have $(aQ_1+bQ_2+cQ_3)(z,w,t)\not=0.$
  Suppose in addition that   for every  $(a,b,c)\in\C^3\setminus\{0\}$ with $a+b+c=0,$  two polynomials $P$ and $aQ_1+bQ_2+cQ_3$  are coprime.
Then every  irreducible component of the hypersurface $\{aQ_1+bQ_2+cQ_3\}$ is  sent  by $f$ onto a hypersurface.
\end{cor}
\begin{proof}
 In order to get a contradiction, suppose  that $\mathcal{G}$ is  an irreducible component of the hypersurface $\{aQ_1+bQ_2+cQ_3\}$ and  $f(\mathcal{G})$  is a point $p\in\P^2,$ where $(a,b,c)\in\C^3\setminus\{0\}$ with $a+b+c=0.$ Using the explicit formula  (\ref{eq2_example}),
the second  hypothesis  ensures  that there  exists
$$[z_0:w_0:t_0]\in  \big (\{aQ_1+bQ_2+cQ_3\}\cap \{P=0\}\big )\setminus \I(f).$$
Consequently,  we get  $p=f([z_0:w_0:t_0])=[1:1:1].$
This  implies that  either the map $\P^2\ni [z:w:t]\mapsto [Q_1:Q_2:Q_3]$ sends $\mathcal{G}$ to the point $[1:1:1]$
or $\mathcal{G}\subset\{P=0\}.$
But the former case  contradicts  the first hypothesis  whereas  the latter case contradicts the third hypothesis.
\end{proof}

\begin{cor}\label{cor_criterion3} Suppose that   for every  $(a,b,c)\in\C^3\setminus\{0\},$  two polynomials $P$ and $aQ_1+bQ_2+cQ_3$  are coprime. Suppose in addition that  the $3\times 3$ matrix whose   $j$-th line is
\begin{equation*}
\Big (\frac{\partial(PQ_j-R)}{\partial z}(1,1,1)   \quad \frac{\partial(PQ_j-R)}{\partial w}(1,1,1)\quad
 \frac{\partial(PQ_j-R)}{\partial t}(1,1,1)\Big)
\end{equation*}
has  the rank $\geq 2.$
Then every  irreducible component of the hypersurface $\{P=0\}$ is  sent  by $f^2$ onto a hypersurface.
\end{cor}
\begin{proof}
Let  $F:\ \C^3\rightarrow\C^3$  be given by
$$F=(F_1,F_2,F_3):=(PQ_1-R,PQ_2-R,PQ_3-R).$$ Then  a straightforward computation  shows that the $j$-th component of  $\frac{F\circ F}{P}$  $(1\leq j\leq 3)$ has  the form
\begin{equation*}
 \frac{\partial(PQ_j-R)}{\partial z}(F)\cdot Q_1+\frac{\partial(PQ_j-R)}{\partial w}(F)\cdot Q_2+
 \frac{\partial(PQ_j-R)}{\partial t}(F)\cdot Q_3+ \mathcal{O}(P),
\end{equation*}
where $\mathcal{O}(P)$ is a polynomial  which can be factored by $P.$ Observe that the  proof of the corollary
will be complete if we can show that for any fixed irreducible divisor $S$ of $P,$
the image of $[S=0]$ by  $\frac{F\circ F}{P}$  $(1\leq j\leq 3)$ is  a  curve.
Using the above  formula, this task is  reduced to show that the (not necessarily dominant) rational map of $\P^2$ whose
$j$-th component is
\begin{equation*}
 \frac{\partial(PQ_j-R)}{\partial z}(1,1,1)\cdot Q_1+\frac{\partial(PQ_j-R)}{\partial w}(1,1,1)\cdot Q_2+
 \frac{\partial(PQ_j-R)}{\partial t}(1,1,1)\cdot Q_3,
\end{equation*}
does not map $[S=0]$ to a point. But this   is  always satisfied taking into account the hypothesis.
\end{proof}

Now  we fix  the  degrees of  $P$ and  $Q_1.$
Using the above corollaries, we see  easily that  with  a generic  choice of the coefficients of
$R,$  $P,$ $Q_1,$  $Q_2,$ $Q_3$
such that  relation   (\ref{eq1_example}) holds,    the hypotheses of Corollary \ref{cor_criterion1}, \ref{cor_criterion2} and \ref{cor_criterion3} are  fulfilled. We thus obtain a family of non AS but QAS self-maps.
The characteristic polynomial of  maps in this  family is (see (\ref{eq2_section_preparatory_results}))
\begin{equation*}
P(t):=t^2-(\deg(P)+\deg(Q_1))t+\deg(P).
\end{equation*}

\end{document}